\documentclass[a4,12pt]{amsart}
\usepackage{amssymb,amstext,amsmath,amscd,amsthm,amsfonts,enumerate,graphicx,latexsym}
\usepackage[usenames]{color}
\newtheorem{thm}{Theorem}[section]
\newtheorem{cor}[thm]{Corollary}
\newtheorem{lem}[thm]{Lemma}
\newtheorem{prop}[thm]{Proposition}
\newtheorem*{mthm}{Main Theorem}
\newtheorem*{thm*}{Theorem}
\theoremstyle{definition}
\newtheorem{dfn}[thm]{Definition}
\newtheorem{rem}[thm]{Remark}

\newtheorem{conj}[thm]{Conjecture}
\newtheorem{ex}[thm]{Example}

\newtheorem*{claim*}{Claim}

\theoremstyle{remark}

\numberwithin{equation}{thm}
\def\m{\mathfrak{m}}
\def\n{\mathfrak{n}}
\def\s{\operatorname{s}}

\def\Ker{\operatorname{Ker}}
\def\Im{\operatorname{Im}}
\def\Coker{\operatorname{Coker}}

\def\H{\operatorname{H}}

\def\height{\operatorname{ht}}
\def\F{\operatorname{F}}
\def\Hom{\operatorname{Hom}}
\def\Ext{\operatorname{Ext}}

\def\edim{\operatorname{edim}}
\def\N{\mathbb{N}}

\def\C{\mathbb{C}}
\begin{document}
\setlength{\baselineskip}{15pt}
\title[Brauer--Thrall for totally reflexive modules]{Brauer--Thrall for totally reflexive modules\\
over local rings of higher dimension}
\author{Olgur Celikbas}
\address{Department of Mathematics, 323 Mathematical Sciences Bldg, University of Missouri, Columbia, MO 65211, USA}
\email{celikbaso@missouri.edu}
\author{Mohsen Gheibi}
\address{Faculty of Mathematical Sciences and Computer, Tarbiat Moallem University, Tehran, Iran/School of Mathematics, Institute for Research in Fundamental Sciences (IPM), P.O. Box: 19395-5746, Tehran, Iran}
\email{mohsen.gheibi@gmail.com}
\author{Ryo Takahashi}
\address{Graduate School of Mathematics, Nagoya University, Furocho, Chikusaku, Nagoya 464-8602, Japan/Department of Mathematics, University of Nebraska, Lincoln, NE 68588-0130, USA}
\email{takahashi@math.nagoya-u.ac.jp}
\urladdr{http://www.math.nagoya-u.ac.jp/~takahashi/}
\thanks{The third author was partially supported by JSPS Grant-in-Aid for Young Scientists (B) 22740008 and by JSPS Postdoctoral Fellowships for Research Abroad}
\begin{abstract}
Let $R$ be a commutative Noetherian local ring.
Assume that $R$ has a pair $\{x,y\}$ of exact zerodivisors such that $\dim R/(x,y)\ge2$ and all totally reflexive $R/(x)$-modules are free.
We show that the first and second Brauer--Thrall type theorems hold for the category of totally reflexive $R$-modules.
More precisely, we prove that, for infinitely many integers $n$, there exists an indecomposable totally reflexive $R$-module of multiplicity $n$. Moreover, if the residue field of $R$ is infinite, we prove that there exist infinitely many isomorphism classes of indecomposable totally reflexive $R$-modules of multiplicity $n$.
\end{abstract}
\maketitle
\tableofcontents
\section{Introduction}

In the 1940s, Brauer \cite{B} and Thrall \cite{T} presented a couple of conjectures, which were later recorded by Jans \cite{J}, known as the {\em first and second Brauer--Thrall conjectures}:

\begin{conj}[Brauer--Thrall]\label{bt}
Let $A$ be a finite-dimensional algebra over a field $k$.
Suppose that there are infinitely many isomorphism classes of finitely generated indecomposable $A$-modules.
Then the following statements hold:
\begin{enumerate}[(1)]
\item
For infinitely many integers $n$, there exists an indecomposable $A$-module of length $n$.
\item
Assume $k$ is infinite. Then, for infinitely many integers $n$, there exist infinitely many isomorphism classes of indecomposable $A$-modules of length $n$.
\end{enumerate}
\end{conj}

\noindent
The Brauer--Thrall conjectures became one of the main problems in representation theory of algebras.
The first Brauer--Thrall conjecture, Conjecture \ref{bt}(1), was proved by Ro\u{\i}ter \cite{R} in 1968. Bautista \cite{Ba} and Bongartz \cite{Bo} proved the second conjecture, Conjecture \ref{bt}(2), in 1985, under the extra hypothesis that $k$ is algebraically closed.

In the 1960s, Auslander \cite{A} introduced the notion of a totally reflexive module as a module of G-dimension zero.
This notion was studied in detail from various points of view and substantial progress was made in the study of totally reflexive modules. In particular, Christensen, Piepmeyer, Striuli and the third author \cite{gcov} proved that if a commutative Noetherian non-Gorenstein local ring $R$ has a nonfree totally reflexive module, then there exist infinitely many isomorphism classes of indecomposable totally reflexive $R$-modules.
(Here, the assumption of non-Gorensteinness is crucial; in fact, simple hypersurface singularities are counterexamples without that assumption.)
It has gradually turned out that totally reflexive modules and homological dimensions associated to them, such as G-dimension, should make a Gorenstein analogue of the classical homological algebra, which is so-called {\em Gorenstein homological algebra} \cite{C,CFH}.

It is quite natural to ask whether there are analogues of the Brauer--Thrall conjectures for totally reflexive modules or not. Motivated by such a question, Christensen, Jorgensen, Rahmati, Striuli and Wiegand \cite{CJRSW} proved the following result:

\begin{thm}[Christensen--Jorgensen--Rahmati--Striuli--Wiegand]\label{five}
Let $(R,\m,k)$ be an Artinian commutative local ring with $\m^3=0$ having embedding dimension $e\ge3$.
Suppose that $R$ admits a pair $\{x,y\}$ of exact zerodivisors.
\begin{enumerate}[\rm(1)]
\item
There exists a family $\{M_n\}_{n\in\N}$ of nonisomorphic indecomposable totally reflexive $R$-modules of length $ne$.
Moreover the minimal free resolution of each $M_n$ is periodic of period at most $2$.
\item
Assume $k$ is algebraically closed. There exists, for each $n\in\N$, a family $\{M_n^\lambda\}_{\lambda\in k}$ of nonisomorphic indecomposable totally reflexive $R$-modules of length $ne$.
Moreover the minimal free resolution of each $M_n^\lambda$ is periodic of period at most $2$.
\end{enumerate}
\end{thm}

\noindent
Here a pair $\{x,y\}$ of exact zerodivisors is by definition a pair of elements in $\m$ that satisfy $(0:x)=(y)$ and $(0:y)=(x)$.

The purpose of this paper is to prove a higher-dimensional counterpart of Theorem \ref{five}.
The main result of this paper is the following:

\begin{mthm}
Let $(R,\m,k)$ be a Noetherian commutative local ring and let $\{x,y\}$ be a pair of exact zerodivisors of $R$. Assume that $R/(x,y)$ has Krull dimension at least $2$ and that every totally reflexive $R/(x)$-module is free.
Let $s$ and $t$ be the multiplicities of the local rings $R/(x)$ and $R/(y)$, respectively.
\begin{enumerate}[\rm(1)]
\item
For every integer $r\ge0$, there exists an exact sequence of $R$-modules 
$$
0 \to (R/(x))^r \to M \to R/(y) \to 0,
$$
where $M$ is indecomposable.
In particular, for each $r\ge0$, there exists an indecomposable totally reflexive $R$-module of multiplicity $rs+t$ whose minimal free resolution is periodic of period at most $2$.
\item
For every integer $r\ge1$ and for every element $a\in R$, there exists an exact sequence of $R$-modules
$$
0 \to (R/(x))^r \to M_a \to R/(y) \to 0, 
$$
where $M_a$ is indecomposable. Furthermore, if $p,q\in R$ satisfies $M_p\cong M_q$, then $\overline{p}=\overline{q}$ in $k=R/\m$.
In particular, if $k$ is infinite, then for each $r\ge1$, there are infinitely many isomorphism classes of indecomposable totally reflexive $R$-modules of multiplicity $rs+t$ whose minimal free resolutions are periodic of period at most $2$.
\end{enumerate}
\end{mthm}

Here are several remarks we should make.

\begin{rem}
(1) The existence of infinitely many totally reflexive modules is not merely enough to prove Brauer--Thrall type theorems for totally reflexive modules.
There is a critical difference between:
\begin{enumerate}[(a)]
\item
the existence of infinitely many totally reflexive modules, and
\item
the Brauer--Thrall type statements that hold for totally reflexive modules.
\end{enumerate}
Indeed, as stated above, property (a) is always satisfied if the ring is a commutative Noetherian non-Gorenstein local ring possessing a nonfree totally reflexive module.
For property (b) the conditions on the multiplicities are essential.

(2) In \cite[Theorem in Introduction]{Holm} Holm proves that if $\{x,y\}$ is an {\em orthogonal} (that is, $(x)\cap(y)=0$) pair of exact zerodivisors in a commutative Noetherian local ring $R$, and $a \in R$ is an element which is regular on $R/(x,y)$, then there exists an infinite family $G_{a}, G_{a^2}, G_{a^3}, \dots, H_{a}, H_{a^2}, H_{a^3}, \dots$ of pairwise nonisomorphic indecomposable totally reflexive $R$-modules whose minimal free resolutions are periodic of period $2$.
One can see that the multiplicities of those modules are all equal to a fixed number, the sum of the multiplicities of $R/(x)$ and $R/(y)$. Thus Holm's theorem does not say anything about Brauer--Thrall for totally reflexive modules.

(3) A more general result than Theorem \ref{five} is proved in \cite[Theorem 3.1]{CJRSW}, which does not assume that the condition $\m^3=0$ holds.
We should notice here that this result does not imply our Main Theorem.
In fact, Main Theorem does not require the elements $x,y$ to be outside $\m^2$.
So, for example, let $A$ be a commutative Noetherian local ring of Krull dimension at least $2$ over which every totally reflexive module is free (e.g. a formal power series ring over a field in two variables), and let $R=A[[X,Y]]/(X^2Y)$ be a residue class ring of a formal power series ring over $A$.
Then, setting $x=Y$ and $y=X^2$, we have a pair of exact zerodivisors $\{x,y\}$ in $R$.
The ring $R/(x,y)=A[[X]]/(X^2)$ has Krull dimension at least $2$, and all totally reflexive modules over the ring $R/(x)=A[[X]]$ are free by \cite[Corollary 4.4]{greg}.
Thus our Main Theorem can be applied to the ring $R$, and it follows that $R$ satisfies the Brauer--Thrall type properties.
In contrast to this, the result \cite[Theorem 3.1]{CJRSW} seems not be able to be applied to this ring $R$ because $y=X^2$ is in $\m^2$.

(4) The assertion of Main Theorem is independent of that of Theorem \ref{five}, though there are several similarities. In particular, for Theorem \ref{five}, if $R$ is not Gorenstein, then every totally reflexive $R/(x)$-module is free, and the multiplicities of $R/(x)$ and $R/(y)$ are equal to $e$ (see \cite[(2.3)]{psit} and \cite[(4.2)]{CJRSW}) so that the number $ne$ in Theorem \ref{five} essentially equals the number $rs+t$ in our Main Theorem. On the other hand, it should be remarked that both Theorem \ref{five}(2) and the theorem of Bautista and Bongartz, discussed above, assume the field $k$ is algebraically closed, while the second statement of our Main Theorem does not require this assumption. The hypotheses of Theorem \ref{five} force the ring considered to be non-Gorenstein, but Main Theorem can be applied to Gorenstein rings. Moreover, each required indecomposable module in our Main Theorem is obtained as a structurally simple module; it is given by only one extension of $(R/(x))^r$ and $R/(y)$.

(5) Examples of rings satisfying the hypotheses of our Main Theorem are abundant. Furthermore the conclusion of Main Theorem does not necessarily hold in case $R/(x,y)$ has Krull dimension less than $2$; we discuss this restriction and provide several examples in Section 4.
\end{rem}

\section{Preliminary Results}

Throughout the rest of this paper $(R,\m,k)$ denotes a commutative Noetherian local ring.
This section is devoted to stating some fundamental definitions and properties which we will freely use in the subsequent sections.
We start by recalling the definition of an exact zerodivisor, that was introduced in \cite{HS}.

\begin{dfn}
An element $x\in\m$ is called an {\em exact zerodivisor} of $R$ if there exists an element $y\in\m$ such that $(0:x)=(y)$ and $(0:y)=(x)$.
Then $y$ is also an exact zerodivisor of $R$, and we say that $\{x,y\}$ is a {\em pair of exact zerodivisors} of $R$.
\end{dfn}

\noindent
For elements $x,y\in\m$, the condition that $\{x,y\}$ is a pair of exact zerodivisors of $R$ is equivalent to the condition that the sequence $\cdots\xrightarrow{x}R\xrightarrow{y}R\xrightarrow{x}R\xrightarrow{y}R\xrightarrow{x}\cdots$ is exact.

Here are some examples of pairs of exact zerodivisors. Note that the first two rings are Gorenstein, while the last two are not.
\begin{ex}
Let $k$ be a field.
\begin{enumerate}[(1)]
\item
Let $R=k[[x]]/(x^n)$ with $n\ge2$.
Then $\{x,x^{n-1}\}$ is a pair of exact zerodivisors of $R$.
\item
Let $R=k[[x,y]]/(xy)$.
Then $\{x,y\}$ is a pair of exact zerodivisors of $R$.
\item
Let $R=k[[x,y,z]]/(x^2,y^2,yz,z^2)$.
Then $\{x,x\}$ is a pair of exact zerodivisors of $R$.
\item
Let $R=k[[x,y,z]]/(x^2-yz,y^2-xz,z^2,xy)$.
Then $\{x,y\}$ and $\{z,z\}$ are pairs of exact zerodivisors of $R$.
\end{enumerate}
\end{ex}

Next we recall the definition of a totally reflexive module.

\begin{dfn}
A finitely generated $R$-module $M$ is called {\em totally reflexive} if the natural homomorphism $M\to M^{\ast\ast}$ is an isomorphism and $\Ext_R^i(M,R)=0=\Ext_R^i(M^\ast,R)$ for all $i>0$, where $(-)^\ast=\Hom_R(-,R)$.
\end{dfn}

\noindent 
We record some of the basic facts on totally reflexive modules.

\begin{rem} \
\begin{enumerate}[(1)]
\item All finitely generated free $R$-modules are totally reflexive.
\item $R$ is Cohen--Macaulay if and only if all totally reflexive $R$-modules are maximal Cohen--Macaulay.
\item $R$ is Gorenstein if and only if the totally reflexive $R$-modules are precisely the maximal Cohen-Macaulay $R$-modules.
\item The totally reflexive $R$-modules form a {\em resolving subcategory} in the category of finitely generated $R$-modules \cite[(3.11)]{ABr}.
Hence, total reflexivity is preserved under taking direct summands, extensions and syzygies.
\item Let $x$ be an exact zerodivisor of $R$. Then the $R$-module $R/(x)$ is totally reflexive.
\end{enumerate}
\end{rem}

Following \cite{greg}, we introduce the definition below.

\begin{dfn}
$R$ is called {\em G-regular} if all totally reflexive $R$-modules are free.
\end{dfn}

\noindent
We note some properties of G-regular local rings, see \cite{greg} for more details.

\begin{rem}\
\begin{enumerate}[(1)]
\item $R$ is regular if and only if $R$ is Gorenstein and G-regular.
\item Assume $R$ is not Gorenstein. If $R$ is {\em Golod}, e.g. Cohen--Macaulay with minimal multiplicity, then $R$ is G-regular. See \cite[(3.5)]{AM} and \cite[(2.5)]{Y}.
\item Assume $R$ is a standard graded algebra over a field such that $\m^3=0$. 
Then $R$ is G-regular provided that the Hilbert series of $R$ is different from $1+et+(e-1)t^2$, where $e=\edim R$ \cite[(3.1)]{Y}.
\end{enumerate}
\end{rem}

\section{Brauer--Thrall type theorems}

In this section we prove our Main Theorem that was recorded in the introduction; first 
we investigate the structure of certain short exact sequences.

\begin{prop}\label{41723}
Let $x\in\m$ and let $I$ be a proper ideal of $R$. Assume $R/I$ is a totally reflexive $R$-module, $x$ is an exact zerodivisor and $R/(x)$ is a G-regular local ring.
Then every exact sequence $0 \to (R/(x))^n \to M \to R/I \to 0$ with $n\ge0$, as an $R$-complex, has a direct summand isomorphic to an exact sequence $0 \to (R/(x))^t \to N \to R/I \to 0$ for some $t$ with $0\le t\le n$ such that $N$ is indecomposable.
In particular, $M\cong(R/(x))^{n-t}\oplus N$.
\end{prop}

\begin{proof}
There is nothing to prove if $M$ is indecomposable, so let us assume $M\cong X\oplus Y$ for some nonzero $R$-modules $X,Y$.
Then $n\ge1$ (since $R/I$ is indecomposable as an $R$-module) and we have an exact sequence $0 \to (R/(x))^n \to X\oplus Y \xrightarrow{(f,g)} R/I \to 0$.
Notice $Y$ is totally reflexive as an $R$-module.
There are elements $v\in X$ and $w\in Y$ with $f(v)+g(w)=\overline1$ in $R/I$.
As $R/I$ is a local ring, either $f(v)$ or $g(w)$ is a unit of $R/I$.
Hence either $f$ or $g$ is surjective.
We may assume that $f$ is so, and we put $Z=\Ker f$.
Then the following commutative diagram with exact rows and columns is obtained:
$$
\begin{CD}
@. 0 @. 0 @. 0 \\
@. @VVV @VVV @VVV \\
\alpha:0 @>>> Z @>>> X @>f>> R/I @>>> 0 \\
@. @VVV @V\binom{1}{0}VV @V{=}VV \\
\beta:0 @>>> (R/(x))^n @>>> X\oplus Y @>(f,g)>> R/I @>>> 0 \\
@. @VVV @VVV @VVV \\
\gamma:0 @>>> Y @>=>> Y @>>> 0 @>>> 0 \\
@. @VVV @VVV @VVV \\
@. 0 @. 0 @. 0
\end{CD}
$$
The left column implies that $Y$ is an $R/(x)$-module.
By \cite[Corollary]{S}, $Y$ is also totally reflexive as an $R/(x)$-module.
As $R/(x)$ is G-regular, we have $Y\cong(R/(x))^r$ for some $0<r\le n$, and hence the left column splits. Thus there is an isomorphism $\beta\cong\alpha\oplus\gamma$ of $R$-complexes, and we have an exact sequence $0 \to (R/(x))^{n-r} \to X \to R/I \to 0$, where $Z\cong(R/(x))^{n-r}$.
Iterating this procedure, we deduce the required conclusion.
\end{proof}

For elements $x,y,a_1,\dots,a_n\in R$ we define an $(n+1)\times(n+1)$ matrix
$$
T(x,y,a_1,\dots,a_n)=\left(\begin{array}{ccc|c}
x & & & a_1 \\
& \ddots & & \vdots \\
& & x & a_n \\
\hline
& & & y
\end{array}\right).
$$
When $n=0$, we regard $T(x,y,a_1,\dots,a_n)$ as the $1\times1$ matrix $(y)$.

To prove our main results, we prepare two lemmas.
The first one is straightforward.
The second one follows from the first one together with the horseshoe and snake lemmas.

\begin{lem}
Suppose that there is a commutative diagram
$$
\begin{CD}
0 @>>> M @>\binom{1}{0}>> M\oplus N @>(0,1)>> N @>>> 0 \\
@. @VfVV @V{\left(\begin{smallmatrix}
f & h \\
0 & g
\end{smallmatrix}\right)}VV @VgVV \\
0 @>>> M' @>>\binom{1}{0}> M'\oplus N' @>>(0,1)> N' @>>> 0
\end{CD}
$$
of $R$-modules.
Let $\delta: \Ker g \to \Coker f$ be the map induced by the snake lemma.
Then $\delta=0$ if and only if $h(\Ker g)\subseteq\Im f$.
\end{lem}

\begin{lem}\label{2126}
Let $x,y\in R$ and $n\ge0$.
The following are equivalent for an $R$-module $M$.
\begin{enumerate}[\rm(1)]
\item
There is an exact sequence $0 \to (R/(x))^n \to M \to R/(y) \to 0$ of $R$-modules.
\item
There is an exact sequence $R^{n+1} \xrightarrow{T(x,y,a_1,\dots,a_n)} R^{n+1} \to M \to 0$ of $R$-modules with $a_1,\dots,a_n\in((x):(0:y))$.
\end{enumerate}
\end{lem}

For a finitely generated $R$-module $M$, we denote by $\nu_R(M)$ its minimal number of generators.
We introduce the following invariant for a pair of elements.

\begin{dfn}
Given $x,y\in\m$, we define the number $\s(x,y)$ to be the supremum of the nonnegative integers $n$ such that there exists an ideal $I$ of $R$ with $\nu_{R/(x,y)}(I(R/(x,y)))=n$.
Note by definition that $0\le\s(x,y)\le\infty$.
\end{dfn}

The number $\s(x,y)$ is related to the decomposability of a module appearing in the middle of a short exact sequence as in our Main Theorem. We study this property in the following two results, the latter of which is our first main result.

\begin{prop}
Let $x,y\in\m$ be elements.
For any integer $n>\s(x,y)$ and any exact sequence $0 \to (R/(x))^n \to M \to R/(y) \to 0$ of $R$-modules, $M$ has a direct summand isomorphic to $R/(x)$.
In particular, $M$ is decomposable.
\end{prop}

\begin{proof}
Lemma \ref{2126} gives an exact sequence $R^{n+1} \xrightarrow{T(x,y,a_1,\dots,a_n)} R^{n+1} \to M \to 0$.
As $n>\s(x,y)$, we must have $\nu_{R/(x,y)}(\overline{a_1},\dots,\overline{a_n})<n$.
We may assume that there are $b_2,\dots,b_n\in R$ such that $\overline{a_1}=\overline{b_2a_2}+\cdots+\overline{b_na_n}$ in $R/(x,y)$.
We have $a_1=cx+dy+b_2a_2+\cdots+b_na_n$ for some $c,d\in R$, and get equivalences of matrices:

$$
{\tiny
\begin{array}{rl}
&\left(\begin{array}{c|ccc|c}
x & & & & a_1 \\
\hline
& x & & & a_2 \\
& & \ddots & & \vdots \\
& & & x & a_n \\
\hline
& & & & y
\end{array}\right)
=
\left(\begin{array}{c|ccc|c}
x & & & & cx+dy+b_2a_2+\cdots+b_na_n \\
\hline
& x & & & a_2 \\
& & \ddots & & \vdots \\
& & & x & a_n \\
\hline
& & & & y
\end{array}\right)
\smallskip
\\\\
&\cong
\left(\begin{array}{c|ccc|c}
x & & & & b_2a_2+\cdots+b_na_n \\
\hline
& x & & & a_2 \\
& & \ddots & & \vdots \\
& & & x & a_n \\
\hline
& & & & y
\end{array}\right)
\cong
\left(\begin{array}{c|ccc|c}
x & -b_2x & \cdots & -b_nx & \\
\hline
& x & & & a_2 \\
& & \ddots & & \vdots \\
& & & x & a_n \\
\hline
& & & & y
\end{array}\right)
\cong
\left(\begin{array}{c|ccc|c}
x & & & & \\
\hline
& x & & & a_2 \\
& & \ddots & & \vdots \\
& & & x & a_n \\
\hline
& & & & y
\end{array}\right).
\end{array}
}
$$
\vspace{0.1in}

\noindent Letting $N$ be the cokernel of the $R$-linear map defined by the $n\times n$ matrix $T(x,y,a_2,\dots,a_n)$, we observe $M\cong R/(x)\oplus N$.
\end{proof}

We can now prove the first main result of this paper:

\begin{thm}\label{2202}
Let $\{x,y\}$ be a pair of exact zerodivisors of $R$.
Assume that the local ring $R/(x)$ is G-regular.
Then for any integer $0\le n\le\s(x,y)$ there exists an exact sequence $$0 \to (R/(x))^n \to M \to R/(y) \to 0$$ of $R$-modules such that $M$ is indecomposable.
\end{thm}

\begin{proof}
The assertion is trivial when $n=0$, so let $n>0$.
As $n\le\s(x,y)$, we have $\nu_{R/(x,y)}(\overline{a_1},\dots,\overline{a_n})=n$ for some $a_1,\dots,a_n\in R$.
Let $M$ be the cokernel of the $R$-linear map defined by $T(x,y,a_1,\dots,a_n)$.
Note that $((x):(0:y))=((x):(x))=R$.
By Lemma \ref{2126} we have an exact sequence
\begin{equation}\label{2136}
0 \to (R/(x))^n \to M \to R/(y) \to 0.
\end{equation}
Suppose that the $R$-module $M$ is decomposable.
Then it follows from Proposition \ref{41723} that there is an exact sequence $0 \to (R/(x))^{n-1} \to N \to R/(y) \to 0$ of $R$-modules which is isomorphic to a direct summand of \eqref{2136} as an $R$-complex.
Lemma \ref{2126} shows that there is an exact sequence $R^n \xrightarrow{T(x,y,b_1,\dots,b_{n-1})} R^n \to N \to 0$.
Since $M\cong R/(x)\oplus N$, it is seen that $M$ has two presentation matrices $T(x,y,a_1,\dots,a_n)$ and $T(x,y,0,b_1,\dots,b_{n-1})$, both of which are $(n+1)\times(n+1)$ matrices.
Considering the $n$th Fitting invariant of $M$ \cite[Page 21]{BH}, we have an equality $(x,y,a_1,\dots,a_n)=(x,y,b_1,\dots,b_{n-1})$ of ideals of $R$. Taking the image in $R/(x,y)$, we see that $\nu_{R/(x,y)}(\overline{a_1},\dots,\overline{a_n})\le n-1$.
This contradiction implies that $M$ is indecomposable.
\end{proof}

The following elementary lemma will be used in the proof of Theorem \ref{41801}.

\begin{lem}\label{41753}
Let $n$ be a positive integer.
Let $x_1,\dots,x_n,y\in\m$ be elements such that $\nu_R(x_1,x_2,\dots,x_n,y)=n+1$.
Let $p,q\in R$ be elements satisfying the equality $(x_1+py,x_2,\dots,x_n)=(x_1+qy,x_2,\dots,x_n)$ of ideals.
Then one has $\overline p=\overline q$ in $k=R/\m$.
\end{lem}

\begin{proof}
Considering the image in $R/(x_2,\dots,x_n)$, we may assume $n=1$.
Then the  equality $(x_1+py)=(x_1+qy)$ holds, and we have $x_1+py=u(x_1+qy)$ for some $u\in R$.
Hence $(1-u)x_1+(p-uq)y=0$, which implies that $1-u$ and $p-uq$ are in $\m$, as $\nu_R(x_1,y)=2$.
Thus $p-q=(p-uq)-(1-u)q\in\m$, and we conclude that $\overline p=\overline q$ in $R/\m$.
\end{proof}

Next is the second main result of this paper:

\begin{thm}\label{41801}
Let $\{x,y\}$ be a pair of exact zerodivisors such that the local ring $R/(x)$ is G-regular.
Then for any integer $0<n<\s(x,y)$ there exists a family $$\{0 \to (R/(x))^n \to M_r \to R/(y) \to 0\}_{r\in R}$$ of exact sequences of $R$-modules such that each $M_r$ is indecomposable and that if $p,q\in R$ satisfies $M_p\cong M_q$, then $\overline p=\overline q$ in $k$.
\end{thm}

\begin{proof}
There are elements $a_1,\dots,a_n,b\in R$ such that $\nu_{R/(x,y)}(\overline{a_1},\dots,\overline{a_n},\overline{b})=n+1$.
For each element $r\in R$, let $M_r$ be the cokernel of the $R$-linear map defined by the matrix $T(x,y,a_1+rb,a_2,\dots,a_n)$.
Note that $\nu_{R/(x,y)}(\overline{a_1+rb},\overline{a_2},\dots,\overline{a_n})=n$.
The proof of Theorem \ref{2202} implies that $M_r$ is an indecomposable $R$-module admitting an exact sequence of $R$-modules of the form $0 \to (R/(x))^n \to M_r \to R/(y) \to 0$.
Let $p,q\in R$ be elements with $M_p\cong M_q$.
Then, taking the $n$th Fitting invariants of $M_p$ and $M_q$, we see that $(x,y,a_1+pb,a_2,\dots,a_n)=(x,y,a_1+qb,a_2,\dots,a_n)$.
This induces an equality $(\overline{a_1+pb},\overline{a_2},\dots,\overline{a_n})=(\overline{a_1+qb},\overline{a_2},\dots,\overline{a_n})$ of ideals of $R/(x,y)$. Now Lemma \ref{41753} shows that $\overline p=\overline q$ in $k$.
\end{proof}

We denote the analytic spread of a proper ideal $I$ of $R$ by $\lambda(I)$.
This equals, by definition, the dimension of the fiber cone $\F(I)=\bigoplus_{n\ge0}I^n/\m I^n$. For all $n\ge0$, the Hilbert function $\H(\F(I),n)$ equals $\nu_R(I^n)$ so that the numerical function $\nu_R(I^n)$ is of polynomial type of degree $\lambda(I)-1$; see \cite[(4.1.3)]{BH}.
This observation and our theorems yield the following result:

\begin{cor}\label{2219}
Let $\{x,y\}$ be a pair of exact zerodivisors of $R$ such that $R/(x)$ is G-regular.
Assume that there exists a proper ideal $I$ of $R$ with $\height(I(R/(x,y))\ge2$.
\begin{enumerate}[\rm(1)]
\item
For every integer $r\ge0$, there exists an exact sequence of $R$-modules $0 \to (R/(x))^r \to M \to R/(y) \to 0$, where $M$ is indecomposable.
\item
For every integer $r\ge1$, there exists a family of exact sequences of $R$-modules $\{0 \to (R/(x))^r \to M_a \to R/(y) \to 0\}_{a\in R}$, where each $M_a$ is indecomposable.
Moreover, if $M_p\cong M_q$ for some $p,q\in R$, then $\overline p=\overline q$ in $k$.
\end{enumerate}
\end{cor}

\begin{proof}
It suffices to see, by Theorems \ref{2202} and \ref{41801}, that $\s(x,y)=\infty$.
As the analytic spread of an ideal is more than or equal to its height \cite[(4.6.13)]{BH}, the numerical function $\nu_{R/(x,y)}(I^n(R/(x,y)))$ is of polynomial type of positive degree.
Therefore $\lim_{n\to\infty}\nu_{R/(x,y)}(I^n(R/(x,y)))=\infty$ and hence $\s(x,y)=\infty$.
\end{proof}

We are now able to prove our Main Theorem that is stated in the introduction:

\begin{proof}[Proof of Main Theorem]
We have, by assumption, that $\height(\m(R/(x,y)))\ge2$.
Hence the first assertions in (1) and (2) follow from Corollary \ref{2219} by letting $I=\m$.

Note that $\dim R/(x)=\dim R=\dim R/(y)$ \cite[(3.3)]{AHS}; hence the multiplicities of $R/(x)$ and $R/(y)$, as local rings, equal those as $R$-modules.
This shows that the multiplicities of the $R$-modules $M$ and $M_a$ are equal to $rs+t$.
As the minimal free resolutions of $(R/(x))^r$ and $R/(y)$ are periodic of period at most $2$, so are those of $M$ and $M_a$. This justifies the second assertions in (1) and (2).
\end{proof}

\section{Several consequences of Main Theorem}

In this section we provide several consequences of our results.
We start with giving a sufficient condition for a local ring to satisfy the hypotheses of our Main Theorem.

\begin{prop}\label{831903}
Let $(S,\n)$ be a local ring, $x,y\in\n$ and let $R=S/(xy)$. Assume $S$ is G-regular, $x$ and $y$ are nonzerodivisors of $S$, $x\notin\n^2$ and $\dim S/(x,y)\ge2$. Then $\{x,y\}$ is a pair of exact zerodivisors of $R$, $\dim R/(x,y)\ge2$ and $R/(x)$ is G-regular.
\end{prop}

\begin{proof}
The only nonobvious assertion is that $R/(x)=S/(x)$ is G-regular.
This follows from \cite[(4.6)]{greg}.
\end{proof}

Recall that regular rings are G-regular. Hence the following consequence of Main Theorem follows from Proposition \ref{831903}.

\begin{cor}\label{831915}
Let $(S,\n,k)$ be a local ring, $0\ne y\in\n$, $x\in\n\setminus\n^2$ and let $R=S/(xy)$. Assume $S$ is regular and that $\dim S/(x,y)\ge2$. Then, for each $r\ge1$, there exist indecomposable totally reflexive $R$-modules of multiplicity $r$. If $r\ge2$ and $|k|=\infty$, then there are infinitely many isomorphism classes of such modules.
\end{cor}

\begin{rem}\label{730807}
The assumption that $\dim S/(x,y)\ge2$ in Corollary \ref{831915} (consequently the assumption $\dim R/(x,y)\ge2$ in Main Theorem) is necessary. In fact:
\begin{enumerate}[(1)]
\item
$\C[[x]]/(x^2)$ and $\C[[x,y]]/(xy)$ have only finitely many nonisomorphic indecomposable maximal Cohen--Macaulay modules; see \cite[(9.9)]{Y2}.
\item
$\C[[x,y]]/(x^2)$ and $\C[[x,y,z]]/(xy)$ have countably many nonisomorphic indecomposable maximal Cohen--Macaulay modules whose multiplicities are $1$ or $2$; see \cite[(2.1),(2.2)]{hsccm}.
\end{enumerate}
\end{rem}

We now apply Proposition \ref{831903} to Cohen--Macaulay non-Gorenstein local rings with minimal multiplicity and obtain another useful consequence of our Main Theorem:

\begin{cor}\label{831733}
Let $(S,\n,k)$ be a Cohen--Macaulay non-Gorenstein local ring with minimal multiplicity $e$.
Let $x,y\in\n$ be nonzerodivisors of $S$ with $x\notin\n^2$ and $\dim S/(x,y)\ge2$.
Then $R:=S/(xy)$ is a Cohen--Macaulay non-Gorenstein local ring.
For each $r\ge1$, there exist indecomposable totally reflexive $R$-modules of multiplicity $re$.
If $r\ge2$ and $|k|=\infty$, then there are infinitely many isomorphism classes of such modules.
\end{cor}

\begin{ex}
Let $k$ be a field.\\
(1) Let
$$
S=\frac{k[[v,w,x,y,z]]}{\mathrm{I}_2\left(
\begin{matrix}
v & w & x \\
x & y & z
\end{matrix}\right)}=\frac{k[[v,w,x,y,z]]}{(vy-wx,vz-x^2,wz-xy)}.
$$
Then $S$ is a $3$-dimensional Cohen--Macaulay non-Gorenstein local ring with an isolated singularity (hence $S$ is a normal domain) and minimal multiplicity $3$.
It follows from Corollary \ref{831733} that
$$
R=S/(xy)=\frac{k[[v,w,x,y,z]]}{(vy-wx,vz-x^2,wz,xy)}
$$
is a $2$-dimensional Cohen--Macaulay non-Gorenstein local ring which admits, for every integer $r\ge1$, indecomposable totally reflexive modules of multiplicity $3r$, and infinitely many such modules if $r\ge2$ and $|k|=\infty$.\\
(2) Let $R$ be either of the following rings:
$$
\frac{k[[x_1,\dots,x_n,y_1,\dots,y_m,z]]}{(x_1,\dots,x_n)^2+(z^2)},\quad
\frac{k[[x_1,\dots,x_n,y_1,\dots,y_m,z,w]]}{(x_1,\dots,x_n)^2+(zw)}\quad (n,m\ge2).
$$
Then note that $R$ is a non-Gorenstein Cohen--Macaulay local ring.
Corollary \ref{831733} implies that, for each $r\ge1$, there exist indecomposable totally reflexive $R$-modules of multiplicity $r(n+1)$, and infinitely many such modules if $r\ge2$ and $|k|=\infty$.
\end{ex}

\section*{Acknowledgments}
Part of this work started during a visit of the first and second authors to the University of Nebraska--Lincoln in June 2012.
We are all grateful for the kind hospitality of this university.
We also very much thank Henrik Holm, Osamu Iyama, Roger Wiegand and the anonymous referees for their valuable comments on the manuscript.



\begin{thebibliography}{99}
\bibitem{hsccm}
{\sc T. Araya; K.-i. Iima; R. Takahashi}, On the structure of Cohen--Macaulay modules over hypersurfaces of countable Cohen--Macaulay representation type, {\em J. Algebra} {\bf 361} (2012), 213--224.
\bibitem{A}
{\sc M. Auslander},
Anneaux de Gorenstein, et torsion en alg\`{e}bre commutative, S\'{e}minaire d'Alg\`{e}bre Commutative dirig\'{e} par Pierre Samuel, 1966/67, Texte r\'{e}dig\'{e}, d'apr\`{e}s des expos\'{e}s de Maurice Auslander, Marquerite Mangeney, Christian Peskine et Lucien Szpiro, \'{E}cole Normale Sup\'{e}rieure de Jeunes Filles, {\it Secr\'{e}tariat math\'{e}matique, Paris}, 1967.
\bibitem{ABr}
{\sc M. Auslander; M. Bridger}, Stable module theory, {\em Mem. Amer. Math. Soc.} {\bf 94} (1969).
\bibitem{AHS}
{\sc L. L. Avramov; I. B. Henriques; L. M. \c{S}ega}, Quasi-complete intersection homomorphisms,  {\em to appear in Pure and Applied Mathematics Quarterly}, posted at \texttt{arXiv:1010.2143}.
\bibitem{AM}
{\sc L. L. Avramov; A. Martsinkovsky}, Absolute, relative, and Tate cohomology of modules of finite Gorenstein dimension, {\em Proc. London Math. Soc. (3)} {\bf 85} (2002), no. 2, 393--440.
\bibitem{Ba}
{\sc R. Bautista}, On algebras of strongly unbounded representation type, {\em Comment. Math. Helv.} {\bf 60} (1985), no. 3, 392--399.
\bibitem{Bo}
{\sc K. Bongartz}, Indecomposables are standard, {\em Comment. Math. Helv.} {\bf 60} (1985), no. 3, 400--410.
\bibitem{B}
{\sc R. Brauer}, On the indecomposable representations of algebras, {\em Bull. Amer. Math. Soc.} {\bf 47} (1941), Abstract 334, Page 684.
\bibitem{BH}
{\sc W. Bruns; J. Herzog}, Cohen--Macaulay rings, revised edition, Cambridge Studies in Advanced Mathematics, 39, {\it Cambridge University Press, Cambridge}, 1998.
\bibitem{C}
{\sc L. W. Christensen}, Gorenstein dimensions, Lecture Notes in Mathematics, 1747, {\em Springer-Verlag, Berlin}, 2000.
\bibitem{CFH}
{\sc L. W. Christensen; H.-B. Foxby; H. Holm}, Beyond totally reflexive modules and back: a survey on Gorenstein dimensions, {\em Commutative algebra---Noetherian and non-Noetherian perspectives}, 101--143, {\em Springer, New York}, 2011.
\bibitem{CJRSW}
{\sc L. W. Christensen; D. A. Jorgensen; H. Rahmati; J. Striuli; R. Wiegand}, Brauer--Thrall for totally reflexive modules, {\em J. Algebra} {\bf 350} (2012), 340--373.
\bibitem{gcov}
{\sc L. W. Christensen; G. Piepmeyer; J. Striuli; R. Takahashi}, Finite Gorenstein representation type implies simple singularity, {\em Adv. Math.} {\bf 218} (2008), no. 4, 1012--1026.
\bibitem{HS}
{\sc I. B. Henriques; L. M. \c{S}ega}, Free resolutions over short Gorenstein local rings, {\em Math. Z.} {\bf 267} (2011), no. 3-4, 645--663.
\bibitem{Holm} 
{\sc H. Holm}, Construction of totally reflexive modules from an exact pair of zero divisors,
{\em Bull. London Math. Soc.} {\bf 43} (2011), no. 2, 278--288. 
\bibitem{J}
{\sc J. P. Jans}, On the indecomposable representations of algebras, {\em Ann. of Math. (2)} {\bf 66} (1957), 418--429.
\bibitem{R}
{\sc A. V. Ro\u{\i}ter}, Unboundedness of the dimensions of the indecomposable representations of an algebra which has infinitely many indecomposable representations (Russian), {\em Izv. Akad. Nauk SSSR Ser. Mat.} {\bf 32} (1968), 1275--1282.
\bibitem{S}
{\sc J. J. M. Soto}, Gorenstein quotients by principal ideals of free Koszul homology, {\em Glasg. Math. J.} {\bf 42} (2000), no. 1, 51--54.
\bibitem{psit}
{\sc R. Takahashi}, Some characterizations of Gorenstein local rings in terms of G-dimension, {\em Acta Math. Hungar.} {\bf 104} (2004), no. 4, 315--322.
\bibitem{greg}
{\sc R. Takahashi}, On G-regular local rings, {\em Comm. Algebra} {\bf 36} (2008), no. 12, 4472--4491.
\bibitem{T}
{\sc R. M. Thrall}, On ahdir algebras, {\em Bull. Amer. Math. Soc.} {\bf 53} (1947), Abstract 22, Page 49.
\bibitem{Y2}
{\sc Y. Yoshino}, Cohen--Macaulay modules over Cohen--Macaulay rings, London Mathematical Society Lecture Note Series, 146, {\em Cambridge University Press, Cambridge}, 1990.
\bibitem{Y}
{\sc Y. Yoshino}, Modules of G-dimension zero over local rings with the cube of maximal ideal being zero, {\em Commutative algebra, singularities and computer algebra (Sinaia, 2002)}, 255--273, NATO Sci. Ser. II Math. Phys. Chem., 115, {\em Kluwer Acad. Publ., Dordrecht}, 2003.
\end{thebibliography}
\end{document}